\documentclass[11pt]{amsart}
\usepackage{amssymb,amsthm}
\usepackage{wrapfig,tikz, tikz-cd, adjustbox}
\usetikzlibrary{arrows}

\numberwithin{equation}{section}
\newtheorem{theorem}[equation]{Theorem}
\newtheorem*{theorem*}{Theorem}
\newtheorem{lemma}[equation]{Lemma}
\newtheorem{proposition}[equation]{Proposition}
\newtheorem{corollary}[equation]{Corollary}

\theoremstyle{definition}

\newtheorem{example}[equation]{Example}
\newtheorem{remark}[equation]{Remark}

\theoremstyle{definition}

\DeclareMathOperator{\HH}{HH}

\DeclareMathOperator{\Hom}{Hom}

\newcommand{\Z}{\mathbb{Z}}
\newcommand{\op}{{\sf op}}
\newcommand{\ot}{\otimes}

\newcommand{\ev}{{\sf ev}}
\newcommand{\act}{\cdot}

\newcommand{\scl}{l}
\newcommand{\scr}{r}

\title[Hochschild Cohomology of Twisted Tensor Products]{Hochschild Cohomology\\ of Twisted Tensor Products}

\author{Benjamin Briggs}
\address{Department of Mathematics, University of Utah, Salt Lake City,
UT 84112} 
\email{briggs@math.utah.edu}
\author{Sarah Witherspoon}
\address{Department of Mathematics, Texas A\&M University, College Station,
TX 77843, USA}
\email{sjw@math.tamu.edu}
\thanks{The second author was partially supported by 
NSF grant DMS-1665286.}
\dedicatory{Dedicated to the memory of Ragnar-Olaf Buchweitz.}

\date{May 2, 2020}

\begin{document}

\maketitle

\begin{abstract}
For a tensor product of algebras twisted by a bicharacter, 
we completely describe its Hochschild cohomology, as a Gerstenhaber algebra,
in terms of the Hochschild cohomology of its component parts. This description generalizes a result of Bergh and Oppermann. It allows us to significantly simplify various calculations in the literature, and to compute Hochschild cohomology for a number of new examples.
\end{abstract}

\section{Introduction}

In this paper we consider the tensor product of two graded algebras, with multiplication twisted by a bicharacter. 
When the two factors are augmented, 
Bergh and Oppermann~\cite{BO} completely described 
the Ext algebra as a twisted tensor product of the Ext algebras of the factors. 
Hochschild cohomology, however, is more complicated, and with
the techniques at the time, they were able only to describe a subalgebra
of the Hochschild cohomology ring. 
This subalgebra may be thought of as the ``untwisted part'' of Hochschild cohomology;  
it was left open  how to describe the remaining, truly  twisted, part of
Hochschild cohomology. 
This we do here.
Moreover, we describe the full Gerstenhaber algebra structure
in terms of that of the component algebras.

To be precise, let $R$ and $S$ be algebras over a field $k$, graded by  abelian groups $A$ and $B$ respectively.
A bicharacter $t : A\times B\rightarrow k^{\times}$ determines a  twisted tensor product algebra
$R\ot^t S$, 
as explained in Section~\ref{sec:prelim} below.  The following result combines 
Theorems \ref{Main-theorem}, \ref{cup-prod-thm}, and \ref{thm:G-alg} (all notation is defined in Section~\ref{sec:prelim}).  

\begin{theorem*}
There is an isomorphism of Gerstenhaber algebras
\[
   \HH^*(R\ot^t S) \cong \bigoplus_{a\in A, \ b\in B}
     \HH^*(R, R_{\hat{b}})^a \ot \HH^*(S, {}_{\hat{a}}S)^b .
\]
Let $M$ be an $R$-bimodule graded by $A$ and let $N$ be an $S$-bimodule
graded by $B$.
There is an isomorphism of graded  $\HH^*(R\ot^t S)$-modules
\[
\HH^*(R\ot^t S,M\ot^tN)\cong \bigoplus_{a\in A, \ b\in B}\HH^*(R,M_{\hat{b}})^a\ot \HH^*(S,{}_{\hat{a}}N)^b.
\]
\end{theorem*}

 This description allows us to simplify drastically the calculation of Hochschild cohomology for several classes of algebras, notably the quantum complete intersections (see Section~\ref{sec:examples}). The fact that this decomposition is compatible with the cup product implies that the Hochschild cohomology of twisted tensor products often has an extremely degenerate product (see Corollary \ref{degenerate-cor}). The fact that it is compatible with the Gerstenhaber bracket means it can be used to simplify the often formidable task of computing brackets.

On the two factors appearing in the theorem we define a twisted cup product $\smile_t$ and twisted Gerstenhaber bracket $[\ ,\,]_t$ (coming from a chain level twisted circle product $\circ_t$). These structures seem interesting in their own right, and we only begin to study them here. The rule for combining these twisted structures together, in the two factors appearing in the main theorem, is a generalization of Manin's definition of the tensor product of two Gerstenhaber algebras~\cite{M}.

The proof is elementary, and we give an explicit isomorphism at the level of Hochschild cochain complexes, taking care to match up the various twistings.  This means in particular that the main theorem could be upgraded into a statement about dg algebras. At the same time, we indicate how the isomorphism can be thought of conceptually in terms of ``orbit Hochschild cohomology'' (see Section~\ref{sec:cup}).

\subsection*{Outline} Section \ref{sec:prelim} contains the necessary notation and homological notions. 
The main isomorphism is constructed in Section~\ref{sec:main}, at the level of graded vector spaces. 
In Section~\ref{sec:cup} we handle the cup product and module structure, and then in Section~\ref{sec:bracket} we treat Gerstenhaber brackets, completing the proof of the above theorem.
Finally, in Sections~\ref{sec:examples} and~\ref{sec:examples2} we give examples. After understanding the statement of the main theorem, we recommend that the reader skip to the example sections to see how it is used.

\section{Preliminaries}\label{sec:prelim}

Throughout this paper $k$ is a field (but a commutative ring would be fine if everything in sight is projective over $k$). All unlabeled tensor products and $\Hom$s are taken over $k$. We denote by $R^\ev = R^\op\ot R$ the enveloping algebra of a $k$-algebra $R$.

\subsection*{Twisted tensor products of algebras}

First we recall how the tensor product of two graded algebras can be twisted by a bicharacter connecting their grading groups. 

Let $R$ and $S$ be $k$-algebras that are
graded by abelian groups $A$ and $B$:
\[
   R = \bigoplus_{a\in A} R^a \quad \mbox{ and }
   \quad S = \bigoplus_{b\in B} S^b .
\]
Suppose also that $t: A\times B \rightarrow k^{\times}$ is a bicharacter, so that
\[
  t(a+a',b) = t(a,b)t(a',b), \quad 
   \quad t(a, b+b') = t(a,b)t(a,b'),
\]
whenever $a,a'\in A$ and $b,b'\in B$. For convenience we also use the notation
\[ t(r,s)=t(a,b) \] 
when $r$ is in $R^a$ and $s$ is in $S^b$,
and similar notation for elements in graded modules.

With this data the {\em twisted tensor product} $R\ot^t S$ is by definition the vector space $R\ot S$, with multiplication given by
\[
   (r\ot s )\cdot( r'\ot s') = t(r',s) \,rr'\ot ss'
\]
for homogeneous elements $r,r'\in R$ and $s,s'\in S$.

Note that $R\ot^t S$ is naturally an $A\oplus B$-graded algebra with $(R\ot^t S)^{a,b} = R^a\ot S^b$ for all $a\in A$ and $b\in B$.

\begin{example}\label{exKoszul}
If $A$ and $B$ are both $\Z$ and $t$ is the sign bicharacter $t(a,b)=(-1)^{ab}$, then this construction yields the usual graded tensor product of two graded algebras (i.e.~following the Koszul sign rule).
\end{example}

\begin{example}\label{exqci}
Let $R = k[x]/(x^m)$ and $S=k[y]/(y^n)$ for some positive integers $m,n\geq 2$,
both graded by $\Z$, with $x$ and $y$ in degree $1$.
Let $q\in k^{\times}$ and $t(a,b) =q^{ab}$ for $a,b\in \Z$.
There is a presentation
\[
   R\ot ^t S \cong k\langle x,y\rangle / ( x^m , \ y^n , \ yx-qxy ) .
\]
This is a {\em quantum complete intersection} in two indeterminates.
The construction can be iterated to obtain a quantum complete
intersection in finitely many indeterminates.
When $m=n=2$ and $q$ is not a root of unity,
these are algebras of infinite global dimension whose Hochschild cohomology
is finite dimensional over $k$, as  discovered by
Buchweitz, Green, Madsen, and Solberg~\cite{BGMS}. 
\end{example} 

We will present more examples in the final two sections.

\subsection*{Twisted tensor products of bimodules} 

Continuing in the setting of the last subsection, suppose that $M$ is an $A$-graded $R$-bimodule, and that $N$ is a $B$-graded $S$-bimodule.

The twisted tensor product $M\ot^t N$ is by definition the $A\oplus B$-graded $R\otimes^t S $-bimodule whose underlying graded vector space is $M\otimes N$, with $R\otimes^t S$-action
\begin{equation}\label{bimodule-tens}
(r\otimes s)\cdot (m\otimes n) =  t(m,s)\,rm \otimes sn
\quad \mbox{ and }
   \quad
(m\otimes n)\cdot (r\otimes s) = t(r,n)\,mr \otimes ns
\end{equation}
for homogeneous elements $r\in R$, $s\in S$, $m\in M$ and $n\in N$.

One may consider $R\ot^tS$ as a bimodule over itself in the usual way, and this notation is consistent in that $R\ot^tS$ is indeed the bimodule twisted tensor product of the bimodules $R$ and $S$.

\subsection*{Group actions}

Let $\widehat{A}=\Hom(A,k^\times)$ be the group of linear characters of $A$. 
Since $R$ is graded by $A$, this group acts naturally on $R$ by setting $\rho\cdot r = \rho(a) r$ for $r\in R^a$ and $\rho\in \widehat{A}$.

The bicharacter $t$ induces a homomorphism $B\to \widehat{A}$, which will be denoted $b\mapsto \hat{b}$, and through this homomorphism $B$ acts naturally on $R$.
Explicitly, the automorphism $\hat{b}$ of $R$ is defined by 
\begin{equation}\label{action-def}
\hat{b}(r) = t(a,b) r \quad \text{for all} \quad r\in R^a. 
\end{equation}
Similarly, through the homomorphism $A\to \widehat{B}$, denoted $a\mapsto \hat{a}$, we obtain from $t$ a natural action of $A$ on $S$. 
These actions allow us to twist the structures of $R$-bimodules
and $S$-bimodules as we describe next.

\subsection*{Twisting module structures along automorphisms}

This is a separate (but, we will see, related) use of the word ``twist''. 

Suppose that $\rho$ is a graded $k$-algebra automorphism of $R$, and that $M$ is a graded  $R$-bimodule. We denote by $M_\rho$ the graded  $R$-bimodule obtained from $M$ by twisting the right $R$-module structure along $\rho$. That is to say, the left action of $R$ is unchanged and the right action is the composition
\[
M_\rho\otimes R\xrightarrow{1\otimes \rho} M\otimes R \xrightarrow{\ \mu \ } M,
\]
where $\mu$ is the given right module structure of $M$. In calculations we will need to distinguish the old and new actions of $R$, so we will use $\act_\rho$ to denote the twisted action and plain concatenation for the original action:
\begin{equation}\label{action-convention}
    m\act_\rho r= m\rho(r).
\end{equation}
Similarly, one may twist the left module structure along $\rho$ to obtain a bimodule ${}_\rho M$.

In particular, any $R$-bimodule $M$ can be twisted by an element of $B$ to produce a new bimodule $M_{\hat{b}}$. And any $S$-bimodule $N$ can be twisted by an element of $A$ to produce a new bimodule ${}_{\hat{a}}N$. It is these twists which appear in the main theorem. 

\subsection*{The bar 
construction and Hochschild cohomology}

Mostly in order to fix notation, we quickly recap the usual construction of the Hochschild cochain complex, and some of the structure that it enjoys.

Denote by 
$BR$ the unreduced \emph{bar construction}, which is a $\Z$-graded vector space with $B_m R = R^{\otimes m}$. We use the bar notation $r_1\otimes \cdots \otimes r_m=[r_1|\cdots |r_m]$.
The bar construction comes with a differential $b_R\colon B_nR\to B_{n-1}R$ given by $b_R[r_1|\cdots |r_m]=\sum_{i=1}^{m-1}(-1)^i [r_1|\cdots | r_{i}r_{i+1}|\cdots |r_m]$. 

The bar construction can be used to build resolutions. In particular, the \emph{bar resolution} of $R$ is by definition the complex $R\ot BR\ot R$ with the $R$-bilinear differential given by $\partial( 1\ot [r_1|\cdots |r_{m}]\ot 1)=$
\[
r_1\ot [r_2|\cdots |r_{m}]\ot 1 + 1\ot b_R[r_1|\cdots |r_{m}] \ot 1+ (-1)^{m} 1\ot [r_1|\cdots |r_{m-1}]\ot r_{m}.
\]
Using this differential we get a free resolution $R\ot BR\ot R\xrightarrow{\ \simeq\ } R$ of $R$-bimodules. In Section \ref{sec:bracket} we will use the short-hand $B(R)=R\ot BR\ot R$ for the bar resolution.

If $M$ is an $R$-bimodule, 
the unreduced Hochschild cochain complex  $C^*(R,M)$ is by definition $\Hom(BR,M)$. 
Its differential is inherited from the bar resolution by way of the isomorphism $\Hom(BR,M)\cong \Hom_{R^\ev}(R\ot BR\ot R,M)$. Explicitly, if $f\in C^m(R,M)$ then $\partial(f) [r_1|\cdots |r_{m+1}]=$
\begin{equation}\label{HHdiff}
 r_1f[r_2|\cdots |r_{m+1}] + fb_R[r_1|\cdots |r_{m+1}] + (-1)^{m+1}f[r_1|\cdots |r_m]r_{m+1}
\end{equation}
for $r_1,\ldots, r_{m+1}$ in $R$. 
The homology of $C^*(R,M)$ is the Hochschild cohomology $\HH^*(R,M)$ of $R$ with coefficients in $M$.

From the $A$-grading of $R$ and $M$, each of $BR$ and $C^*(R,M)$  and $\HH^*(R,M)$ inherit an $A$-grading. 

All of the above applies just as well to $S$ with its $B$-grading, and indeed $R\ot^tS$ with its $A\oplus B$-grading.

\subsection*{The twisted bar resolution}
One may tensor together the bar resolutions of $R$ and $S$, with action
as in equation~(\ref{bimodule-tens}), to obtain (by the K\"unneth Theorem) a quasi-isomorphism
\[
(R\ot BR\ot R)\ot^t(S\ot BS\ot S)\xrightarrow{\ \simeq\ } R\ot^t S.
\]
The left-hand side has its usual tensor product differential here; the twist only affects the bimodule structure.

Now, there is a unique $R\ot^t S$-bimodule isomorphism
\[
(R\ot BR\ot R)\ot^t(S\ot BS\ot S) \cong (R\ot ^t S)\ot BR\ot BS\ot (R\ot^t S)
\]
such that 
\[
(1 \ot [r_1|\cdots |r_m]\ot 1)\ot (1\ot [s_1| \cdots | s_n ]\ot 1) \ \mapsto\  (1\ot 1)\ot [r_1|\cdots |r_m]\ot [s_1| \cdots | s_n ]\ot (1\ot 1) .
\]
If we write out in full what happens to the differential under this isomorphism we obtain the following proposition.

\begin{proposition}\label{diff-prop}
$(R\ot ^t S)\ot BR\ot BS\ot (R\ot^t S)$ is naturally an $R\ot^t S $-bimodule resolution of $R\ot^t S $ when equipped with the following differential:

\[
 \partial\ \left((1\ot 1)\ot [r_1|\cdots |r_m]\ot [s_1| \cdots | s_n ]\ot (1\ot 1)
  \right) = 
 \]
 \begin{align*}
&   ( r_1\ot 1)\ot [r_2|\cdots |r_m]\ot [s_1| \cdots | s_n ]\ot (1\ot 1) +\\ 
  &(1\ot 1)\ot b_R[r_1|\cdots |r_m]\ot [s_1| \cdots | s_n ]\ot (1\ot 1) +\\
  (-1)^m t(r_m, b)^{-1}&( 1\ot 1)\ot [r_1|\cdots |r_{m-1}]\ot [s_1| \cdots | s_n ]\ot (r_m\ot 1) + \\
(-1)^mt(a, s_1)^{-1}&( 1\ot s_1)\ot [r_1|\cdots |r_m]\ot [s_2| \cdots | s_n ]\ot (1\ot 1) +\\ 
   (-1)^m&(1\ot 1)\ot [r_1|\cdots |r_m]\ot b_S[s_1|\cdots | s_n ]\ot (1\ot 1) +\\
 (-1)^{m+n} &(1\ot 1)\ot [r_1|\cdots |r_{m}]\ot [s_1| \cdots | s_{n-1} ]\ot (1\ot s_n).\\
\end{align*}
Here $a$ is the $A$-degree of $[r_1|\cdots |r_m]$ and $b$ is the $B$-degree of $[s_1| \cdots | s_n ]$.
\end{proposition}

Moving on, we also will need to use the diagonal map
\begin{equation}\label{eqn:diagonal}
   \Delta\colon  BR\ot BS \longrightarrow BR\ot BS\ot BR\ot BS
\end{equation}
from \cite{GNW}, 
which is by definition given by
\begin{align*}
\Delta \ [ & r_1  | \cdots |  r_{m} ] \otimes 
  [s_1 | \cdots | s_n ] = \\ 
&   \sum_{i,j}^{} (-1)^{(m-i)j}
t ( a_i,b_j)^{-1} \, [r_1 | \cdots | r_i]
    \otimes [ s_1 | \cdots | s_j ] \otimes 
 [r_{i+1} | \cdots | r_{m} ] \otimes 
     [s_{i+1} | \cdots | s_n ] ,
\end{align*}
the sum over $1\leq i \leq m$, $1\leq j\leq n$, 
where $a_i$ is the $A$-degree of $[r_{i+1}|\cdots |r_m]$, and $b_j$ is the $B$-degree of $[s_1| \cdots | s_j ]$.

\subsection*{The complex $C^*_t(R,S,M,N)$ and its structure} 
The differential described in Proposition \ref{diff-prop} above induces a dual differential on
\begin{align*}
\Hom(BR\otimes BS,\, &M\otimes N) \\ 
 \cong &\  \Hom_{(R\ot^t S)^{\ev}}((R\ot ^t S)\ot BR\ot BS\ot (R\ot^t S),M\ot^tN).
\end{align*}
We will write $C^*_t(R,S,M,N)$ for the complex $\Hom(BR\otimes BS,M\otimes N) $ equipped with this differential. By construction, the homology of this complex is the $A\oplus B$-graded Hochschild cohomology
$\HH^*(R\ot^t S,M\ot^tN)$.

Taking $M=R$ and $N=S$ the complex  $C^*_t(R,S,R,S)$ has a natural product
\[
f\smile g = \mu (f\otimes g) \Delta
\]
where $\mu $ is the product on $R\otimes^tS$ and $\Delta$ is the diagonal map (\ref{eqn:diagonal}). In \cite{GNW} it is explained why this computes the usual cup product on $\HH^*(R\otimes^tS,R\otimes^tS)$. In fact, one can check that this makes $C^*_t(R,S,M,N)$ into a dg algebra quasi-isomorphic to the usual Hochschild cochain algebra of $R\otimes^t S$.

Let us introduce some notation which will be useful for the rest of the paper. With enough finiteness conditions (see the beginning of Section~\ref{sec:main})
any element of $C^*_t(R,S,M,N)$ can be split into a sum of parallel tensor products of maps $BR \to M$ and $BS\to N$. We use the square $\boxtimes$ as our notation for this parallel tensor product, so
\begin{equation}\label{box-notation}
    (f\boxtimes g)([r]\otimes [s])=
(-1)^{mn}\ f[r]\otimes g[s]
\end{equation}
for elements $f\in C^m(R,M)$, $g\in C^n(S,N)$, $[r]= [r_1|\cdots |r_m]\in B_mR$ and $[s]=[s_1|\cdots |s_n]\in B_nS$. With this notation, expanding the cup product explicitly on elements gives the following lemma. 

\begin{lemma}\label{cup-prod-lemma} Take elements $f\in C^m(R,R)^a$, $f'\in C^{m'}(R,R)^{a'}$, $g\in C^n(S,S)^b$ and $g'\in C^{n'}(S,S)^{b'}$. Then the cup product on $C^*_t(R,S,R,S)$ satisfies
\begin{align*}
\big((f\boxtimes g)  & \smile (f'\boxtimes g')\big) \ \big([r_1|\cdots | r_{m+m'}]\otimes [s_1|\cdots | s_{n+n'}]\big) \\
& = (-1)^i \cdot t^*\cdot f[r_1|\cdots | r_m]f'[r_{m+1}|\cdots |r_{m+m'}] \otimes g[s_1|\cdots | s_n]g'[s_{n+1}|\cdots |s_{n+n'}]
\end{align*}
\begin{align*}
\text{where}\quad\quad & i = mn'+nn'+m'n' +mm'+mn\\
\text{and}\ \  \quad\quad & t^* =  t(a',b)t(a',s)t(r',b).
\end{align*}  
Here $r'$ is the $A$-degree of $[r_{m+1}|\cdots |r_{m+m'}]$ and $s$ is the $B$-degree of $[s_1|\cdots | s_n]$.
\end{lemma}

\section{Main Result}\label{sec:main}

\subsection*{}\label{finiteness-cond} 
First we impose some finiteness conditions: we require each component $R^a$ of $R$ to be finite dimensional, and we require that the set ${\rm supp}(R)\subseteq A$ on which $R$ is nonzero satisfies the condition that $n$-fold addition ${\rm supp}(R)^{\times n}\to A$ has finite fibers for all $n$. This just means that $R^{\otimes n}$ is degree-wise finite dimensional for each $n$. Certainly if $R$ is finite dimensional (in total) then there is nothing to worry about. We assume the same for $S$ and its $B$-grading.

Using the notation~(\ref{box-notation}), we define a map of $\Z\oplus A \oplus B$-graded vector spaces
\[
   \phi \colon \Hom(BR,M)\otimes \Hom(BS,N)\longrightarrow \Hom(BR\otimes BS,M\otimes N)
\] 
\[
f\otimes g \quad \mapsto\quad  f\boxtimes g.
\]

\begin{theorem}\label{Main-theorem}\label{vect-space-theorem}
Under the just stated finiteness conditions, $\phi$ induces an isomorphism
\[
\HH^*(R\ot^t S,M\ot^tN)\cong \bigoplus_{a\in A, \ b\in B} \ \HH^*(R,M_{\hat{b}})^a\ot \HH^*(S,{}_{\hat{a}}N)^b
\]
of graded vector spaces.
\end{theorem}

In the next two sections we refine the theorem by describing cup products and Gerstenhaber brackets
on Hochschild cohomology of $\HH^*(R\ot^t S)$ in terms of this direct sum decomposition
when $M=R$ and $N=S$. 

\begin{remark}
Let us note the difference between graded algebras and \emph{gradable} algebras. In some contexts the groups $A$ and $B$ may be considered only as auxiliary data allowing one to encode the twist which constructs $R\otimes^t S$. Then one might be interested in the ordinary $\Z$-graded Hochschild cohomology, computed by first forgetting the $A\oplus B$-grading. This can be a subtle issue with real consequences (Hochschild cohomology need not commute with forgetting a grading), but under our finiteness assumptions one \emph{can} safely do this by using Theorem \ref{Main-theorem} to compute the $A\oplus B$-graded Hochschild cohomology and then forgetting the grading. 
\end{remark}

\begin{proof}[Proof of Theorem \ref{vect-space-theorem}.]
Since each of $BR,BS,M$ and $N$ is degree-wise finite dimensional by hypothesis, $\phi$ is an isomorphism of $\Z\oplus A \oplus B$-graded vector spaces. Our task is only to show that
\begin{equation}\label{eqn:phi}
\phi\colon C^*(R,M_{\hat{b}})^a\otimes C^*(S,{}_{\hat{a}}N)^b\xrightarrow{\ \cong\ } C^*_t(R,S,M,N)^{a,b}
\end{equation} 
is a chain map. This will be a slightly tedious matter of checking that all the twists match up.

Take $f$ in $C^m(R,M_{\hat{b}})^a$ and $g$ in $C^n(S,{}_{\hat{a}}N)^b$.

We first compute $\phi\partial (f\otimes g)=\phi\partial(f)\otimes g+(-1)^m \phi f\otimes \partial(g)$. We can split this computation into two cases corresponding to the two summands here. Specifically, we can evaluate this: (\ref{case1}) on elements of the form $[r_1|\cdots |r_{m+1}]\otimes [s_1|\cdots|s_{n}]$; and (\ref{case2}) on elements of the form $[r_1|\cdots |r_{m}]\otimes [s_1|\cdots|s_{n+1}]$.

We first deal with case (\ref{case1}). In this calculation we use the short-hand notation $[r]=[r_1|\cdots |r_{m+1}]$ and   $[s]=[s_1|\cdots|s_{n}]$. We use our formula (\ref{HHdiff}) for the Hochschild differential. Also remember our action notation from equation~(\ref{action-convention}). Note that $\partial(g)[s]=0$, so
\begin{align}\tag{I}\label{case1}
\phi \partial (f\otimes g) ( [r]\otimes [s]) = &
-(-1)^m(-1)^{(m+1)n}\  r_1 f[r_2|\cdots |r_{m+1}]\otimes g[s]   \\
& -(-1)^m (-1)^{(m+1)n} \ fb_R[r]\otimes g[s]  \nonumber \\
& -(-1)^m (-1)^{(m+1)n} (-1)^{m+1} \ f[r_1|\cdots |r_m]\act_{\hat{b}} r_{m+1}\otimes g[s]  \nonumber \\
= &
-(-1)^{mn+m+n}\  r_1 f[r_2|\cdots |r_{m+1}]\otimes g[s]   \nonumber \\
& -(-1)^{mn+m+n} \ fb_R[r]\otimes g[s]   \nonumber \\
& -(-1)^{mn+m+n}(-1)^{m+1} t(r_{m+1},b) \ f[r_1|\cdots |r_m]r_{m+1}\otimes g[s] . \nonumber 
\end{align}
Similarly in case (\ref{case2}), with $[r]=[r_1|\cdots |r_m]$ and $[s]=[s_1|\cdots|s_{n+1}]$, we have $\partial(f)[r]=0$ so
\begin{align}\tag{II}\label{case2}
\phi \partial (f\otimes g) ( [r]\otimes [s]) = &
-(-1)^m (-1)^n(-1)^{m(n+1)}\   f[r]\otimes s_1\act_{\hat{a}} g[s_2|\cdots |s_{n+1}]   \\
& -(-1)^m (-1)^n(-1)^{m(n+1)}\ f[r]\otimes gb_S[s]   \nonumber \\
& -(-1)^m (-1)^n(-1)^{m(n+1)}(-1)^{n+1}  \ f[r]\otimes g[s_1|\cdots|s_n] s_{n+1}  \nonumber \\
= &
-(-1)^{mn+n}t(a,s_1)\   f[r]\otimes s_1 g[s_2|\cdots |s_{n+1}]   \nonumber\\
&-(-1)^{mn+n} \ f[r]\otimes gb_S[s]   \nonumber \\
& -(-1)^{mn+n}(-1)^{n+1} \ f[r]\otimes g[s_1|\cdots|s_n] s_{n+1} . \nonumber 
\end{align}

Next we calculate $\partial(\phi(f\ot g))([r]\ot [s])$ in the same two cases.

In case (\ref{case1}) we take $[r]=[r_1|\cdots |r_{m+1}]$ in $(B_{m+1}R)^{a'}$ and   $[s]=[s_1|\cdots|s_{n}]$ in $(B_nS)^{b'}$. We use Proposition \ref{diff-prop} to compute the differential on $C^*_t(R,S,M,N)$, and as in our previous calculation the choice of $[r]$ and $[s]$ means that half the terms described in Proposition \ref{diff-prop} vanish.
\begin{align}\tag{\ref{case1}}
\partial(\phi(f\ot g))([r]\ot [s]) = & -(-1)^{m+n}(-1)^{mn}\ (r_1\ot 1) f[r_2|\cdots|r_{m+1}]\ot g[s] \\
& -(-1)^{m+n}(-1)^{mn}\  fb_R[r]\ot g[s] \nonumber\\
& -(-1)^{m+n}(-1)^{mn}(-1)^{m+1} t(r_{m+1},b')^{-1}\  f[r_1|\cdots |r_m]\ot g[s] (r_{m+1}\ot 1)\ \nonumber\\
= & 
-(-1)^{mn+m+n}\  r_1 f[r_2|\cdots |r_{m+1}]\otimes g[s]  \nonumber \\
& -(-1)^{mn+m+n} \ fb_R[r]\otimes g[s]   \nonumber \\
& -(-1)^{mn+m+n}(-1)^{m+1} t(r_{m+1},b) \ f[r_1|\cdots |r_m]r_{m+1}\otimes g[s] .  \nonumber 
\end{align}
The last line comes from the fact that 
\[
t(r_{m+1},b')^{-1} f[r_1|\cdots |r_m]\ot g[s] (r_{m+1}\ot 1)=t(r_{m+1},b')^{-1}t(r_{m+1},b+b' )f[r_1|\cdots |r_m]r_{m+1}\ot g[s] ,
\] 
since $g[s]$ has degree $b+b'$ in $B$, and then $t(r_{m+1},b')^{-1}t(r_{m+1},b+b' )=t(r_{m+1},b )$. The result matches exactly the calculation of $\phi \partial (f\otimes g) ( [r]\otimes [s])$ in case (\ref{case1}).

It remains to compute $\partial(\phi(f\ot g))([r]\ot [s])$ in case (\ref{case2}), with $[r]=[r_1|\cdots |r_m]$ in $(B_mR)^{a'}$ and $[s]=[s_1|\cdots|s_{n+1}]$ in $(B_{n+1}S)^{b'}$.
\begin{align}\tag{\ref{case2}}
\partial(\phi(f\ot g))([r]\ot [s]) = & -(-1)^{m+n}(-1)^{mn}(-1)^m t(a',s_{n+1})^{-1} \ (1\ot s_1) f[r]\ot g[s_2|\cdots|s_{n+1}]  \\
& -(-1)^{m+n}(-1)^{mn}(-1)^m\  f[r]\ot gb_S[s] \nonumber\\
& -(-1)^{m+n}(-1)^{mn}(-1)^{m+n+1} \  f[r]\ot g[s_1|\cdots|s_n] (1\ot s_{n+1})\ \nonumber\\
= & 
-(-1)^{mn+n} t(a,s_{n+1})\  f[r]\ot s_1 g[s_2|\cdots|s_{n+1}]
\nonumber \\
& -(-1)^{mn+n} \ f[r]\otimes gb_S[s]   \nonumber \\
& -(-1)^{mn+n}(-1)^{n+1}  \ f[r]\ot g[s_1|\cdots|s_n]s_{n+1} . \nonumber 
\end{align}
This matches the computation of $\phi \partial (f\otimes g) ( [r]\otimes [s])$ in case (\ref{case2}), so we are done.
\end{proof}

\section{The Cup Product}\label{sec:cup}

In this section and the next we explain how to handle to the algebraic structure in Theorem~\ref{Main-theorem}. But first we make a construction which seems to be interesting in its own right.

\subsection*{Orbit Hochschild cohomology}

The following construction applies to any algebra with a group action, but we may as well use the same notation as above and consider the (right) action of the group $B$ on $R$. The orbit Hochschild cohomology of $R$ is by definition  $\bigoplus_{b\in B} \HH^*(R, R_{\hat{b}})$.

We can define a cup product on this graded vector space 
by the following formula, in terms of the standard Hochschild cochain complex. If $f\in C^m(R, R_{\hat{b}})$ and  $f'\in C^{m'}(R, R_{\hat{b'}})$, 
then we set 
\begin{equation}\label{eqn:twisted-cup1}
(f\smile_t f')\ [r_1|\cdots|r_{m+m'}] = (-1)^{mm'}f[r_1|\cdots|r_m] \,\hat{b}(f'[r_{m+1}|\cdots|r_{m+m'}]),
\end{equation}
where the automorphism $\hat{b}$ is given by~(\ref{action-def}). 
Then $f\smile_t f'$ is an element of $C^{m+m'}(R, R_{\widehat{b+b'}})$.

One can think of this product as follows: from the $B$-action on $R$ we form an ``orbit endomorphism ring'' $\bigoplus_{b} \HH^*(R, R_{\hat{b}}) = \bigoplus_b \Hom_{D(R^\ev)}(R,R_{\hat{b}})$. Just as for ordinary Hochschild cohomology the product can be defined using the (derived) tensor product $\otimes_R$ on $D(R^\ev)$, but it must involve a twist. The automorphism $\hat{b}$ can naturally be considered as a map $R_{\hat{b'}}\to {}_{\hat{b}}R_{\hat{b}\hat{b'}}$, and the composition
\[
R\cong R\otimes_R R \xrightarrow{f\otimes_R f'} R_{\hat{b}}\otimes_R R_{\hat{b'}} \xrightarrow{1\otimes_R \hat{b}} R_{\hat{b}}\otimes_R {}_{\hat{b}}R_{\hat{b}\hat{b'}}\cong R_{\hat{b}\hat{b'}}
\]
is the map $f\smile_tf'\colon R\to R_{\hat{b}\hat{b'}}$.

It is straightforward to check that the product is associative. Moreover, orbit Hochschild cohomology is always twisted commutative in the following sense. 

\begin{proposition}\label{R-twisted-comm-prop}\label{degenerate-prod}
The product (\ref{eqn:twisted-cup1}) satisfies
\begin{align*}
    f\smile_t f'   
    & = (-1)^{mm'}t(a',b) f'\smile_t f 
\end{align*}
whenever $f\in \HH^m(R,R_{\hat{b}})^a$ and $f'\in \HH^{m'}(R,R_{\hat{b'}})^{a'}$. In particular,
\[
    f\smile_t f' = 0 \ \text{ unless }\  t(a',b)=t(a,b')^{-1} .
\]
\end{proposition}

Note that this implies commutativity with respect to the braided monoidal structure given by $t$.

The second statement above is useful in calculations. It will mean that the Hochschild cohomology of a twisted tensor product often has a very degenerate cup product.

\begin{proof}[Sketch of proof]
One can do this with a twisted version of the usual Eckmann-Hilton argument~\cite{SA}, working in $D(R^\ev)$. But it will be useful for us later to introduce the twisted circle product, so we follow Gerstenhaber's original argument.

Given  $f\in C^m(R,R_{\hat{b}})$ and $f'\in C^{m'}(R,R_{\hat{b'}})$ we define an element $f'\circ_tf$ in $ C^{m+m'-1}(R,R_{\widehat{b+b'}})$ according to the rule
\begin{align*}
    (f'\circ_t f)&[r_1|\cdots|r_{m+m'-1}]= \\
    & \sum_{i}  (-1)^{i(m'+1)} 
 f'[ r_1|\cdots | r_{i} | f
   [ r_{i+1}| \cdots | r_{i+m'}]| \hat{b} r_{i+m'+1}
  |\cdots |  \hat{b}r_{m+m'-1} ].
\end{align*}
With this, a computation reveals that
\[
\partial(f'\circ_t f)+\partial(f')\circ_t f+(-1)^{m'}f'\circ_t \partial(f) =  f\smile_t (\hat{b}^{-1}f'\hat{b})-  (-1)^{mm'}f'\smile_t f. 
\]
So at the level of Hochschild cohomology,
\[
f'\smile_t f = (-1)^{mm'}f\smile_t(\hat{b}^{-1}f'\hat{b}).
\]
Another computation shows that 
$\hat{b}^{-1}f'\hat{b}=t(a',b)^{-1}f'$. The second statement follows from the first by exchanging 
$f$ and $f'$ twice.
\end{proof}

All of this applies just as well to the (left) action of $A$ on $S$. There is a natural product on the orbit Hochschild cohomology  $\bigoplus_{a\in A}\HH^*(S,{}_{\hat{a}}S)$ given by $g\smile_t g' = \hat{a}g\smile g'$, that is
\begin{equation}\label{eqn:twisted-cup2}
(g\smile_t g')[s_1|\cdots|s_{n+n'}] = (-1)^{nn'}\hat{a}(g[s_1|\cdots|s_n])\, \,g'[s_{m+1}|\cdots|s_{n+n'}]
\end{equation}
for $g\in C^n(S,{}_{\hat{a}}S)^b$ and $g'\in C^{n'}(S,{}_{\hat{a}'}S)^{b'}$. As before this is associative and twisted commutative in the following sense.
\begin{proposition}\label{S-twisted-comm-prop}
The product (\ref{eqn:twisted-cup2}) satisfies
\[
g\smile_t g' = (-1)^{nn'}t(a,b') g'\smile_t g
\]
whenever $g\in \HH^n(S,{}_{\hat{a}}S)^b$ and $g'\in \HH^{n'}(S,{}_{\hat{a}'}S)^{b'}$.
\end{proposition}

 The proof uses an analogous twisted circle product for $S$. We record it here for use in the next section:
\begin{align*}
    (g'\circ_t g)& [s_1|\cdots | s_{n+n'-1}] = \\
    & \sum_i (-1)^{i(n'+1)} g'[\hat{a}'s_1|\cdots | \hat{a}'s_i |
    g[s_{i+1} | \cdots | s_{i+n'} ]| s_{i+n'+1} | \cdots | s_{n+n'-1}] .
\end{align*}

\subsection*{The cup product in the main theorem} Because of the \emph{non-twisted} tensor product in the enveloping algebra $(R\otimes^t S)^\op\otimes (R\otimes^t S)$ which underlies Hochschild cohomology, there must be some kind of untwisting appearing in the main theorem. This is manifest in the following cup product on
\[
\big(\bigoplus_{b\in B} \HH^*(R, R_{\hat{b}})\big)\otimes \big( \bigoplus_{a\in A}\HH^*(S,{}_{\hat{a}}S) \big),
\]
which we define by  the formula
\begin{equation}\label{untwistedcupdef}
    (f\otimes g) \smile (f'\otimes g') = (-1)^{m'n}t(a',b)^{-1}(f\smile_t f')\otimes (g\smile_t g')
\end{equation}
where $g\in\HH^n(S,{}_{\hat{a}}S)^b$ and $f'\in\HH^{m'}(R,R_{\hat{b'}})^{a'}$.  This is an ``untwisted cup product'' because the twist $t(a',b)^{-1}$ is the inverse of the expected one.

With this product, the diagonal subalgebra 
\[
\bigoplus_{a\in A,\ b\in B} \HH^*(R, R_{\hat{b}})^a\otimes \HH^*(S,{}_{\hat{a}}S)^b
\]
becomes graded-commutative in the usual sense, because the  twists from Propositions~\ref{R-twisted-comm-prop} and~\ref{S-twisted-comm-prop} cancel out.   Hence, the following statement makes sense.

\begin{theorem}\label{cup-prod-thm}
With the just defined product (\ref{untwistedcupdef}) on the left-hand side, 
when $M=R$ and $N=S$, 
the isomorphism $\phi$ of Theorem~\ref{vect-space-theorem} is one of algebras:
\[
\bigoplus_{a\in A,\ b\in B} \HH^*(R,R_{\hat{b}})^a\ot \HH^*(S,{}_{\hat{a}}S)^b\xrightarrow{\ \cong \ }\HH^*(R\ot^t S).
\]
\end{theorem}

In order to establish fully that the second isomorphism in Theorem~\ref{Main-theorem} is
one of $\HH^*(R\ot^t S)$-modules, 
we should also define an action of $\HH^*(R\ot^t S)$ on $\bigoplus_{a,b} \HH^*(R,M_{\hat{b}})^a\ot \HH^*(S,{}_{\hat{a}}N)^b$ at the chain level, and show that $\phi$, defined in~(\ref{eqn:phi}), respects this action. Since this is formally extremely similar to the case $M=R$ and $N=S$ described above, without any further interesting features, we omit the proof. 

Combining Theorem~\ref{cup-prod-thm} and Proposition~\ref{degenerate-prod} yields the following consequence. 

\begin{corollary}\label{degenerate-cor}
If $x\in \HH^*(R\ot^t S,R\ot^t S)^{a,b}$ and $x'\in \HH^*(R\ot^t S,R\ot^t S)^{a',b'}$, then
\[
    x\smile x' = 0 \ \text{ unless }\  t(a',b)=t(a,b')^{-1}.
\]
\end{corollary}

So, despite $\HH^*(R\ot^t S,R\ot^t S)$ being graded commutative, the fact that it is built out of twisted commutative algebras implies that its product is often very degenerate.

\begin{proof}[Proof of Theorem~\ref{cup-prod-thm}]
We will check that the product defined in this section agrees with the expression obtained in Lemma~\ref{cup-prod-lemma}. Let $f,f',g$ and  $g'$ be as in the statement of Lemma~\ref{cup-prod-lemma}. Firstly
\begin{equation}\label{cup-prod-thm-pf-eqn}
\phi\big((f\otimes g) \smile (f'\otimes g') \big) = (-1)^{m'n}t(a',b)^{-1}(f\smile_t f')\boxtimes (g\smile_t g'),
\end{equation}
so we compute 
\begin{align*}
(f\smile_t &\, f')  \boxtimes  (g\smile_t g') \ [r_1|\cdots | r_{m+m'}]\otimes [s_1|\cdots | s_{n+n'}] =\\
&   (-1)^{i} \cdot f[r_1|\cdots|r_m] \,\hat{b}(f'[r_{m+1}|\cdots|r_{m+m'}]) \otimes \hat{a}(g[s_1|\cdots|s_n])\, \,g'[s_{m+1}|\cdots|s_{n+n'}]
\end{align*}
where $i = mm'+nn'+(m+m')(n+n')$. This is by defintion
\[
(-1)^{i}\cdot t^*\cdot  f[r_1|\cdots | r_m]f'[r_{m+1}|\cdots |r_{m+m'}] \otimes g[s_1|\cdots | s_n]g'[s_{n+1}|\cdots |s_{n+n'}]
\]
where $t^*= t(a'+r',b)t(a,b+s)$, with $r'$ being the $A$-degree of $[r_{m+1}|\cdots |r_{m+m'}]$ and $s$ being the $B$-degree of $[s_1|\cdots | s_n]$.

If we multiply this by the scalar $(-1)^{m'n}t(a',b)^{-1}$ appearing in \eqref{cup-prod-thm-pf-eqn}, then we find that our coefficient matches that of Lemma~\ref{cup-prod-lemma}, so we are done.
\end{proof}

\section{The Gerstenhaber Bracket}\label{sec:bracket}

We start by defining a bracket on $\bigoplus_{b\in B} \HH^*(R, R_{\hat{b}})$. Once again we point out that this applies to any algebra with a group action. 
If $f\in C^m (R, R_{\hat{b}})^{a}$ and $f'\in C^{m'}(R, R_{\hat{b'}})^{a'}$, then by definition
\[
   {[} f , f' {]}_t = t^{-1}(a,b')f\circ_t f' - (-1)^{(m-1)(m'-1)} 
      f'\circ_t   f ,
\]
where the twisted circle product $\circ_t$ was defined in Section~\ref{sec:cup}.

Similarly, when $g\in C^n(S, {}_{\hat{a}}S)^b$ and $g\in C^{n'}(S, {}_{\hat{a}'}S)^{b'}$, we define
\[
   [g,g']_t = g\circ_t g' - (-1)^{(n-1)(n'-1)} t^{-1}(a',b) g'\circ_t g .
\]
We refrain from discussing the sense in which these brackets make $\bigoplus_{b\in B} \HH^*(R, R_{\hat{b}})$ and $\bigoplus_{a\in A} \HH^*(S, {}_{\hat{a}}S)$ into ``twisted Gerstenhaber algebras''.

In the following theorem we combine the two brackets, adapting and generalizing a definition of Manin~\cite{M} to define a bracket in the setting of twisted tensor products.

\begin{theorem}\label{thm:G-alg}
The isomorphism 
$\phi$ of Theorem~\ref{vect-space-theorem} is one of Gerstenhaber
algebras: 
\[
   \bigoplus_{a\in A, \ b\in B} \HH^*(R, R_{\hat{b}})^a\ot
   \HH^*(S, {}_{\hat{a}}S)^b \xrightarrow{\ \cong\ }
   \HH^*(R\ot^t S),
\] 
where on the left-hand side, the bracket is defined by ${[}f\ot g, f'\ot g' {]} =$
\[
  (-1)^{(m-1)n'} {[}f,f'{]}_t \ot (g\smile_t g')
   + (-1)^{m(n'-1)} (f\smile_t f')\ot {[}g,g'{]}_t .
\]
\end{theorem}

Before proving the theorem, 
we summarize some results from~\cite{GNW,NW} that we will
need for computing brackets. 
Our bracket formula generalizes that given by Le and Zhou~\cite{LZ}, 
who showed that the Hochschild cohomology ring of a tensor product
of algebras is isomorphic, as a Gerstenhaber algebra,
to the graded tensor product of their Hochschild cohomology rings.
Our result is a twisted analogue.

Note that the diagonal map $\Delta$ of~(\ref{eqn:diagonal})  
is coassociative and counital by its definition.
Therefore the Gerstenhaber bracket of $f\boxtimes g$ and $f'\boxtimes g'$,
representing elements of $\HH^*(R\ot^t S)$, 
may be computed as follows~\cite{NW}. 
The circle product  $(f\boxtimes g)\circ (f'\boxtimes g')$
can be taken to be the following composition:
\begin{equation}\label{eqn:circle}
      (f\boxtimes g) \Phi  (1\ot_{R\ot^t S} 
       (f'\boxtimes g')\ot_{R\ot^t S} 1) \Delta^{(2)},
\end{equation} 
 where $\Delta^{(2)}=(\Delta\ot 1)\Delta$, and where  $\Phi$ is the homotopy given in~\cite[Lemma 3.5]{GNW} for twisted
tensor products, in accordance with  
the theory of~\cite{NW}. It is given by
\[
    \Phi = (G_{B(R)}\ot F^l_{B(S)} + F^r_{B(R)}\ot G_{B(S)}) \sigma,
\]
with $G_{B(R)}$, $F^l_{B(S)}$, $F^r_{B(R)}$, $G_{B(S)}$,  $\sigma$
defined as below, and $B(R)$, $B(S)$ the bar resolutions for $R$ and $S$.
Letting $\mu_R\colon  B(R) \rightarrow R$ be the natural quasi-isomorphism, set
\begin{equation}\label{FlFr}
   F^{\scl}_{B(R)} = \mu_R\ot 1_{B(R)}
    \ \ \ \mbox{ and } \ \ \
   F^{\scr}_{B(R)}=1_{B(R)}\ot \mu_R,
\end{equation}
as maps from $B(R)\ot_R B(R)$ to $B(R)$, 
where $1_{B(R)}$ is the identity map on $B(R)$,
and similarly for $B(S)$. The map
$G_{B(R)}\colon  B(R)\ot _R B(R) \rightarrow B(R)[1]$ is defined by 
\begin{equation}\label{equation-G}
   G_{B(R)} ((r_0\ot [r_1 | \cdots r_{p-1}] \ot r_p) \ot_R (1
    \ot [r_{p+1} |\cdots | r_n] \ot 
     r_{n+1}))  \hspace{4cm}
\end{equation}
$$
  \hspace{3cm}   = (-1)^{p-1} r_0\ot [r_1 | \cdots | r_{p-1}
      | r_p | r_{p+1} |\cdots | r_n ] \ot
     r_{n+1} $$
for all $r_0,\ldots, r_{n+1}$ in $R$. 
The chain map 
\[
  \sigma\colon  (B(R)\ot  B(S))\ot _{R\ot ^t S}(B(R)\ot B(S)) \rightarrow
    (B(R)\ot _R B(R)) \ot ^t (B(S)\ot _S B(S))
\]
is an isomorphism of $(R\ot ^t S)^\ev$-modules in each degree given by
\[
  \sigma ((x\ot y)\ot (x'\ot y')) = (-1)^{jp} t(x' , y)
   (x\ot x') \ot (y\ot y')
\]
on homogeneous elements in 
$(B_i(R)\ot  B_j(S)) \ot _{R\ot ^t S} (B_p(R)\ot B_q(S))$.

\begin{proof}[Proof of Theorem~\ref{thm:G-alg}]
Let $f,f',g,g'$ be as before. 
We will compute the bracket $[f\boxtimes g , f ' \boxtimes g ']$ 
on the right side of the isomorphism in the statement of
the theorem by applying it to elements of the form 
$${[} r_1 |\cdots | r_{m''}{]} \ot 
   {[}s_1 | \cdots | s_{n''}{]} \quad
   \text{ in } BR\ot BS$$ 
where
$m'' + n'' = m+m' + n+ n'-1$
 and $r_1,\ldots, r_{m''}\in R$, $s_1,\ldots, s_{n''}\in S$. 
We will first 
compute the circle product~(\ref{eqn:circle}), canonically identifying
$f,g,f',g'$ with the corresponding bimodule homomorphisms: 
\begin{align*}
& ((f \boxtimes g)\circ (f'\boxtimes g')) 
  ((1\ot {[} r_1| \cdots | r_{m''} {]}\ot 1) \ot 
   (1\ot {[} s_1| \cdots | s_{n''} {]}\ot 1)) \\
& = (f\boxtimes g) \left(G_{B(R)}\ot F^\scl_{B(S)} + 
    F^\scr_{B(R)}\ot G_{B(S)} \right) \sigma 
 (1 \ot (f'\boxtimes g') \ot 1 ) \Delta^{(2)} \\
 & \hspace{4cm} ((1\ot {[} r_1| \cdots | r_{m''} {]}\ot 1)\ot 
   (1\ot {[} s_1| \cdots | s_{n''}{]}\ot 1)) \\ 
& = (f\boxtimes g) \left(G_{B(R)} \ot F^\scl_{B(S)} + 
     F^\scr_{B(R)}\ot G_{B(S)} \right) \sigma 
 (1 \ot (f'\boxtimes g') \ot 1) (\Delta\ot 1) \\
 & \hspace{1.5cm} \Big( \sum_{j=0}^{m''} \sum_{i=0}^{n''} (-1)^{i(m''-j)} 
  t^{-1}({[}r_{j+1}|\cdots | r_{m''}{]} , {[}s_1| \cdots | s_i{]}) \\
 & \hspace{2cm} 
  ((1\ot {[}r_1 |\cdots | r_j{]} \ot 1)\ot 
   (1\ot {[} s_1 | \cdots | s_i {]}\ot 1) ) \ot_{R\ot^tS}\\ 
&\hspace{2cm}      ( (1\ot {[} r_{j+1} |\cdots | r_{m''}{]}\ot 1)\ot  
   (1\ot  {[}s_{i+1}|\cdots | s_{n''}{]}\ot 1) ) \Big) \\
&=(f\boxtimes g) \left(G_{B(R)} \ot F^\scl_{B(S)} + 
   F^\scr_{B(R)}\ot G_{B(S)} \right)
  \sigma (1 \ot (f'\boxtimes g')\ot 1)\\
&\quad  \Big( \sum_{j=0}^{m''} \sum_{i=0}^{n''} \sum_{p=0}^i \sum_{l=0}^j 
  (-1)^{i(m''-j)}(-1)^{p (j-l)} t^{-1}({[}r_{j+1} |\cdots | r_{m''}{]} 
       , {[}s_1 | \cdots | s_i{]}) t^{-1}
      ({[}r_{l+1} |\cdots | r_j{]} ,{[} s_1| \cdots | s_p{]}) \\
 & \hspace{2cm} 
   ((1\ot {[} r_1 |\cdots | r_l {]}\ot 1)\ot (1\ot {[} s_1| \cdots | s_p{]}
   \ot 1) ) \ot _{R\ot^tS}\\
&\hspace{2cm}   ((1\ot {[} r_{l+1}|\cdots | r_j{]}\ot 1)
   \ot (1\ot {[} s_{p+1}|\cdots | s_i{]}\ot 1) ) 
  \ot_{R\ot^t S}\\ 
  &\hspace{2cm} ((1\ot {[}r_{j+1}|\cdots | r_{m''} {]}\ot 1)
   \ot (1\ot {[} s_{i+1}| \cdots | s_{n''}{]}\ot 1)) \Big). \\
\end{align*}
We wish to apply $(1\ot (f'\boxtimes g')\ot 1)$ to the above sum, 
so it suffices to consider only terms for which 
$m'= j-l$, $n'=i-p$. The  sign involved is thus
$$
   (-1)^{(p+l)(m'+n')} = (-1)^{ (m'+n')(j-m' + i-n')},
$$
and the above becomes
\begin{align*}
& = (f\boxtimes g) \left(G_{B(R)} \ot F^\scl_{B(S)} + 
   F^\scr_{B(R)}\ot G_{B(S)} \right) \sigma \\
 &\quad  \Big( \sum_{j=m'}^{m''} \sum_{i=n'}^{n''}  (-1)^{i (m''-j)} (-1)^{(i-n')m'} 
    (-1)^{(m'+n')(j-m'+i-n')} \\
 &\hspace{1cm} t^{-1}({[}r_{j+1} |\cdots | r_{m''}{]} , 
     {[}s_1 | \cdots | s_i{]})t^{-1} ({[}r_{j-m'+1} |\cdots | r_j{]},
     {[} s_1 | \cdots | s_{i-n'}{]}) \\
&\hspace{1cm}
  ((1\ot {[}r_1|\cdots | r_{j-m'}{]}\ot 1)\ot (1\ot {[} s_1| \cdots | 
   s_{i-n'}{]}\ot 1)) \ot_{R\ot^tS}\\ 
&\hspace{1cm}   \big(f'({[}r_{j-m'+1}| \cdots | r_j{]}) \ot g'({[} s_{i-n'+1}|
   \cdots | s_i{]}) \big) \ot_{R\ot^tS}\\
&\hspace{1cm}   ((1\ot {[} r_{j+1}| \cdots | r_{m''}{]}\ot 1)
   \ot (1\ot {[} s_{i+1}| \cdots | s_{n''}{]}\ot 1)) \Big). \\
\end{align*}
Now apply the module action, 
and apply $\sigma$ (which comes with the sign $(-1)^{(i-n')(m''-j)}$), 
to obtain
\begin{align*}
& = (f\boxtimes g) \left(G_{B(R)} \ot F^\scl_{B(S)} + 
   F^\scr_{B(R)}\ot G_{B(S)} \right) \sigma \\
 &\quad  \Big( \sum_{j=m'}^{m''} \sum_{i=n'}^{n''}  (-1)^{i (m''-j)} (-1)^{(i-n')m'} 
    (-1)^{(m'+n')(j-m'+i-n')} \\
 &\hspace{.5cm} t^{-1}({[}r_{j+1}|\cdots | r_{m''}{]} , 
   {[}s_1 | \cdots | s_i{]}) t^{-1}( {[}r_{j-m'+1}|\cdots | r_j{]} ,
    {[} s_1 | \cdots | s_{i-n'}{]}) \\
&\hspace{.5cm}  
   t(f'({[} r_{j-m'+1}| \cdots | r_j{]}) , {[}s_1| \cdots | s_{i-n'}{]}) \\   
&\hspace{.5cm}
 ((1\ot {[}r_1|\cdots | r_{j-m'}{]}\ot f'({[} r_{j-m'+1}| \cdots | r_j{]})) \ot 
 (1\ot {[} s_1| \cdots | s_{i-n'}{]} \ot g'({[} s_{i-n'+1}|\cdots | s_i{]})) 
   \ot_{R\ot^tS}\\
&\hspace{.5cm} ((1\ot {[} r_{j+1}| \cdots | r_{m''}{]}\ot 1)
  \ot (1\ot {[} s_{i+1}| \cdots | s_{n''}{]}\ot 1)) \Big) \\
& = (f\boxtimes g) \left(G_{B(R)} \ot F^\scl_{B(S)} + 
   F^\scr_{B(R)}\ot G_{B(S)} \right) \\
 &\quad  \Big( \sum_{j=m'}^{m''} \sum_{i=n'}^{n''}  (-1)^{-n' (m''-j)} (-1)^{(i-n')m'} 
    (-1)^{(m'+n')(j-m'+i-n')} \\
&\hspace{.5cm} t^{-1}({[}r_{j+1}|\cdots | r_{m''}{]} , 
  {[}s_1| \cdots | s_i{]})t^{-1} ({[}r_{j-m'+1}|\cdots | r_j {]},
   {[}s_1 | \cdots | s_{i-n'}{]})\\
&\hspace{.5cm} t(f'({[}r_{j-m'+1}| \cdots | r_j{]}) ,
   {[}s_1| \cdots | s_{i-n'}{]}) t({[}r_{j+1}| \cdots | r_{m''}{]}, 
  1\ot {[} s_1| \cdots | s_{i-n'}{]} \ot 
    g'({[} s_{i-n'+1}| \cdots | s_i{]})) \\     
&\hspace{.5cm}
 (1\ot {[}r_1|\cdots | r_{j-m'}{]}\ot f'({[} r_{j-m'+1}| \cdots | r_j{]})) 
 \ot_R (1\ot {[} r_{j+1}| \cdots | r_{m''}{]}\ot 1) \ot \\
 &\hspace{.5cm} 
 (1\ot {[} s_1| \cdots | s_{i-n'}{]}\ot g'({[} s_{i-n'+1}| \cdots | s_i{]})) 
 \ot_S (1\ot {[}s_{i+1}| \cdots | s_{n''}{]}\ot 1) \Big). \\
\end{align*}
Denote by $t^{**}$ the twisting coefficient in the above equation,
which simplifies to:
\[
   t^{**} = t(a', [s_1|\cdots | s_{i-n'}]) 
   t( [r_{j+1} | \cdots | r_{m''}] , b') .
\]

Next we will apply 
$G_{B(R)} \ot F^\scl_{B(S)} +  F^\scr_{B(R)}\ot G_{B(S)}$,
and there are signs associated to each term. 
In applying $G_{B(R)}\ot F^\scl_{B(S)}$, 
necessarily $i=n'$ for the image to be non-zero, and the sign
is $(-1)^{j-m'}$. 
In applying $F^\scr_{B(R)}\ot G_{B(S)}$, 
necessarily $j=m''$ for the image to be non-zero, and
the sign  is $(-1)^{i-n'}$
with an additional sign $(-1)^{j-m'+m''-j}=(-1)^{m''-m'}=(-1)^m$ (as for this
term, we take $m''=m + m'$) since the degree of the map $G_{B(R)}$ is 1. 
The above expression thus becomes
\begin{align*}
&= (f\boxtimes g) \Big(
  \sum_{j=m'}^{m''}   (-1)^{-n' (m''-j)} 
    (-1)^{(m'+n')(j-m')}(-1)^{j-m'}t([r_{j+1}\cdots | r_{m''}] , b') \\
  & \hspace{.5cm} (1\ot {[} r_1|\cdots | r_{j-m'} | f'
   ( {[} r_{j-m'+1}|\cdots | r_j{]}) | r_{j+1}
  |\cdots | r_{m''} {]}\ot 1) \ot 
  ( g'({[} s_1|\cdots | s_i{]} )\ot [s_{i+1} | \cdots | s_{n''}{]}\ot 1) \\
& \hspace{.5cm}+\sum_{i=n'}^{n''}     (-1)^{(i-n')m'} 
    (-1)^{(m'+n')(m+i-n')}(-1)^{i-n'}(-1)^{m} t(a', 
   [s_1 | \cdots | s_{i-n'}]) \\
&  \hspace{.5cm} (1\ot {[} r_1| \cdots | r_{m}] \ot f'({[} r_{m+1}|\cdots | 
   r_{m''}]))\ot (1\ot   
 {[} s_1|\cdots | s_{i-n'} | g'({[} s_{i-n'+1}
 |\cdots | s_i{]})| s_{i+1}|\cdots | s_{n''}{]}\ot 1) \Big) \\
& =   \sum_{j=m'}^{m''}   (-1)^{-n' (m''-j)} 
    (-1)^{(m'+n')(j-m')}(-1)^{j-m'}\\
  & \hspace{.5cm} f({[} r_1|\cdots | r_{j-m'} | f'
   ({[} r_{j-m'+1}| \cdots | r_j{]})| \hat{b}r_{j+1}
  |\cdots | \hat{b}r_{m''} {]}) \ot 
   \hat{a}'( g'({[}s_1|\cdots | s_i{]}))
g( {[} s_{i+1}| \cdots | s_{n''}{]})\\
& \hspace{.5cm}+\sum_{i=n'}^{n''}     (-1)^{(i-n')m'} 
    (-1)^{(m'+n')(m+i-n')}(-1)^{i-n'}(-1)^{m} \\
&  \hspace{.5cm}f({[} r_1|\cdots | r_{m} {]} )
  \hat{b}( f'({[}r_{m+1}|\cdots | r_{m''}{]}))   \ot 
  g({[}\hat{a}'s_1|\cdots | \hat{a}'s_{i-n'} | g'({[} s_{i-n'+1}
  |\cdots | s_i{]})| s_{i+1} |\cdots | s_{n''}{]}). \\
\end{align*}
We wish to rewrite the sums. The first sum involves $f\circ_t f'$, in
which the term indexed by $j$ has a sign $(-1)^{(m'-1)(j-m')}$.
The second sum involves $g\circ_t g'$, in which the term indexed by $i$
has a sign $(-1)^{(n'-1)(i-n')}$. Accommodating these signs and rewriting, the 
above is equal to 
$$
   (-1)^{n'(m+n-1)} (f\circ_t f')\boxtimes (g ' \smile_t g)
 + (-1)^{m(n'-1)}(f\smile_t f')\boxtimes (g \circ_t g')$$
applied to the input. 
By Proposition~\ref{R-twisted-comm-prop}, reversing the order of
$g'$, $g$ in the first term, we finally find that 
$(f\boxtimes g )\circ (f'\boxtimes g')$ is equal to 
\[
   (-1)^{n'(m-1)} t^{-1}(a,b') (f\circ_t f') \boxtimes (g\smile_t g') +
  (-1)^{m(n'-1)} (f\smile_t f') \boxtimes (g\circ_t g') .
\]
Similarly, 
\begin{align*} 
&(f' \boxtimes\ g')\circ (f\boxtimes g) \\
& \hspace{2cm}=  (-1)^{n(m'+n'-1)} (f'\circ_t f)\boxtimes (g  \smile_t g')
 + (-1)^{m'(n-1)}(f'\smile_t f)\boxtimes (g' \circ_t g). \\
\end{align*} 
By Proposition~\ref{R-twisted-comm-prop}, reversing the order of $f'$, $f$
in the second term, we obtain 
\[
   (-1)^{n(m'+n'-1)} (f'\circ_t f)\boxtimes (g\smile_t g') 
   + (-1)^{m' (m+n-1)}t^{-1}(a',b) (f\smile_t f') \boxtimes
   (g'\circ_t g) .
\]
We thus have found that 
\[
\begin{aligned}
& [ f\boxtimes g , f'\boxtimes g'] \\
& = (f\boxtimes g ) \circ (f'\boxtimes g') - (-1)^{(m+n-1)(m'+n'-1)} 
    (f'\boxtimes g')\circ (f\boxtimes g) \\
& = (-1)^{n'(m-1)} t^{-1}(a,b') (f\circ_t f')\boxtimes (g\smile_t g') 
     + (-1)^{m(n'-1)} (f\smile_t f') \boxtimes (g\circ_t g') \\
& \quad - (-1)^{(m'+n'-1)(m-1)} (f'\circ_t f)\boxtimes (g\smile_t g')
  - (-1)^{(n'-1)(m+n-1)} t^{-1}(a',b) (f\smile_t f')
  \boxtimes  (g'\circ_t g) .
\end{aligned}
\]

We are now ready to compare with the formula given at the start of this
section, for $[f\ot g , f'\ot g']$ on the left side of the claimed
isomorphism in the theorem statement. 
This is
\[
\begin{aligned}
&(-1)^{(m-1)n'} [f,f']_t \ot (g\smile_t g') + (-1)^{m(n'-1)} (f\smile_t f')
   \ot [g,g']_t \\
& = (-1)^{(m-1)n'} \big(t^{-1}(a,b') f\circ_t f'  - (-1)^{(m-1)(m'-1)}
    f'\circ_t f \big) \ot (g\smile_t g') \\
& \quad +(-1)^{m(n'-1)} (f\smile_t f') \ot \big( g\circ_t g'
    - (-1)^{(n-1)(n'-1)} t^{-1}(a',b) g'\circ_t g\big) .
\end{aligned}
\]
Comparison with the previous calculation shows that 
\[
   \phi([f\ot g,f'\ot g']) 
   = [ f\boxtimes g , f'\boxtimes g'] = 
  [ \phi(f\ot g), \phi(f'\ot g')] .\qedhere
\]
\end{proof} 

\begin{remark}
 Bergh and Oppermann's \cite[Theorem 4.6]{BO} is a special case of our results. Their result is recovered by restricting the isomorphism of Theorem \ref{vect-space-theorem} to the subspace graded by $\ker t(-,B)\times \ker t(A,-)\subseteq A\times B$. This identifies exactly the part of $\bigoplus \ \HH^*(R,R_{\hat{b}})^a\ot \HH^*(S,{}_{\hat{a}}S)^b$ on which the twists act trivially.
The Gerstenhaber bracket on $\HH^*(R\otimes^tS)$ was partially computed by Grimley,   Nguyen,  and  the second author in \cite[Theorem 6.3]{GNW} in terms of Bergh and Oppermann's decomposition. It is shown in loc.~cit.~that on the untwisted part of $\HH^*(R\otimes^tS)$ (i.e.~the restriction to $\ker t(-,B)\times \ker t(A,-)$) the bracket can be computed explicitly using the bracket on the two factors. Theorem \ref{thm:G-alg} above extends this to all of Hochschild cohomology, explaining how to account for the twists.
\end{remark}

\section{Quantum complete intersections and iterated twisted products}\label{sec:examples}

In this section we present a series of examples, namely the quantum complete 
intersections, as an application of our main theorem.
We also explain how to extend the theorem to the case of iterated twisted tensor product algebras.

We begin with the case of two indeterminates, 
continuing from Example \ref{exqci} above. 
Fix $q\in k^\times$. We consider the Hochschild cohomology of the quantum complete intersection $\Lambda_q(m,n)=k[x]/(x^m)\otimes^tk[x]/(x^n)$. The two factors are graded by $\Z$, generated in degree $1$, and  $t:\Z\times\Z\to k^\times$ is the bicharacter $t(a,b) = q^{ab}$. 
The Hochschild cohomology was computed for $m=n=2$ in \cite{BGMS}
, and then later for all $m$ and $n$ in \cite[Theorem 3.3]{BE}. In \cite{EF} the cup product was computed when $m=n$ is the order of $q$ in $k^\times$. The Gerstenhaber brackets were computed fully in \cite[Section 5]{GNW} in the case $m=n=2$. 

All of these results can be recovered using the main theorem of this paper, but for the sake of novelty we deal with a new case here, and use Theorem \ref{thm:G-alg} to  calculate the Gerstenhaber brackets on  $\HH^*(\Lambda_m\otimes^t\Lambda_n)$ for all $m$ and $n$. There are many cases to consider, depending on the characteristic of $k$ and the order of $q$ in $k^\times$, and for the sake of brevity we will only consider here the case that $q$ has infinite order. Let us emphasize however that all of the cases can be dealt with readily (one only needs the  patience to write them all out).

\begin{theorem}\label{qcibracketthm}
If $q$ is not a root of unity then as an algebra
\[
\HH^*(\Lambda_q(m,n))\cong  k[U]/(U^2) \times_k {\textstyle \bigwedge^*_k}(V,W) , 
\]
i.e.~it is the fiber product of $k[U]/(U^2)$, $U$ in degree $0$, with an exterior algebra $\bigwedge^*_k(V,W)$, $V,W$ in degree $1$. The bracket is given by
\[
[V,U] = (m-1) U,\quad [W,U] = (n-1) U, \quad \text{and }\ [V,W]=0.
\]
\end{theorem}

\begin{proof}
The algebra structure is known \cite[Theorem 3.3]{BE}, but we give the full calculation to demonstrate Theorem \ref{cup-prod-thm}.

Denote $\Lambda(m)=k[x]/(x^m)$ and take $b\in \Z$. We need to compute $\HH^*(\Lambda(m),\Lambda(m)_{\hat{b}})$. There is a well-known $2$-periodic bimodule resolution of $\Lambda(m)$:
\[
\cdots\xrightarrow{x\otimes 1-1\otimes x}\Lambda(m)^{\ev}[-m]\xrightarrow{\sum_{i=0}^{m-1} x^{m-i-1}\otimes x^{i}}\Lambda(m)^{\ev}[-1]\xrightarrow{x\otimes 1-1\otimes x}\Lambda(m)^{\ev},
\]
where $[i]$ denotes the shift in grading by $i$. Applying $\Hom_{\Lambda(m)}(-,\Lambda(m)_{\hat{b}})$ produces the complex
\[
\cdots\xleftarrow{x(1-q^{b})}\Lambda(m)[m]\xleftarrow{x^{m-1} (\sum_{i=0}^{m-1}q^{ib})}\Lambda(m)[1]\xleftarrow{x(1-q^{b})}\Lambda(m).
\]
From this one can read off the cohomology (assuming that $q^b\neq 1$ when $b\neq 0$):
\begin{equation}
    \HH^i(\Lambda(m),\Lambda(m)_{\hat{b}})=
    \begin{cases}
    \Lambda(m) & \text{if }i=0, b=0\\
    (x^{m-1}) & \text{if }i=0, b\neq 0\\
    (\Lambda(m)/(mx^{m-1}))[\frac{i}{2}m]& \text{if }i>0 \text{ is even}, b=0\\
    {\rm Ann}_{\Lambda(m)}(mx^{m-1})[\frac{i-1}{2}m+1]& \text{if }i>0 \text{ is odd}, b=0\\
    0& \text{if }i>0, b\neq0.\\
    \end{cases}
\end{equation}
The computation for $\HH^j(\Lambda(n),{}_{\hat{a}}\Lambda(n))$ is essentially the same. When one comes to combine these vector spaces according to the decomposition of Theorem \ref{vect-space-theorem}, one finds that almost all of the terms have at least one of the two factors equal to zero. The only surviving terms are
\begin{align*}
    \HH^0(\Lambda_q(m,n))^{0,0}   \ = \ \ & k(1\otimes 1) & \ \\
    \HH^0(\Lambda_q(m,n))^{m-1,n-1} \   = \ \  & k(x^{m-1}\otimes x^{n-1}) \\
    \HH^1(\Lambda_q(m,n))^{0,0} \   = \ \  & k(x[1]\otimes1) + k(1 \otimes x[1]) \\
    \HH^2(\Lambda_q(m,n))^{0,0} \ = \ \ & k (x[1] \otimes x[1]).
\end{align*}
This matches the desired result if we set $U=x^{m-1}\otimes x^{n-1}$ and $V=x[1]\otimes1$ and $W= 1 \otimes x[1]$.

Lastly, $[V,W]=0$ for degree reasons, and using Theorem \ref{thm:G-alg},
\begin{align*}
    [V,U] & =  [x[1]\otimes1,x^{m-1}\otimes x^{n-1}]  \\
    & = [x[1],x^{m-1}]_t\otimes x^{n-1} - x^{m-1}\otimes[x^{n-1},1]_t\\
    & = (m-1)x^{m-1}\otimes x^{n-1} -0 ,
\end{align*}
where as usual, bracketing a degree~1 element with a degree~0 element
amounts to applying the corresponding derivation to the algebra element. 
Similarly $[W,U]=(n-1)x^{m-1}\otimes x^{n-1}$, and this completes the proof.
\end{proof}

Now we point out how the main result of this paper can be extended to the case of iterated twisted tensor products. Suppose we have abelian groups $A_1,...,A_n$ and a collection $t$  of bicharacters $t_{ij}:A_i\times A_{j}\to k^\times$ for $i<j$. If $R_1,...,R_n$ are algebras graded by  $A_1,...,A_n$ respectively, we can form the twisted tensor product $R_1\otimes^t\cdots \otimes^tR_n$. As a graded vector space this is $R_1\otimes \cdots \otimes R_n$, and the multiplication is determined by $(1\otimes\cdots \otimes r_i\otimes\cdots \otimes 1) \cdot (1\otimes\cdots \otimes r_j\otimes\cdots \otimes 1)= $
\begin{align*}
    (1\otimes\cdots \otimes r_ir_j\otimes\cdots \otimes 1) \quad & \text{if } i=j , \\
    (1\otimes\cdots \otimes r_i\otimes\cdots \otimes r_j\otimes\cdots \otimes 1) \quad & \text{if } i<j  , \\
    t_{ji}(r_j,r_i)(1\otimes\cdots \otimes r_j\otimes\cdots \otimes r_i\otimes\cdots \otimes 1) \quad& \text{if } i>j.
\end{align*}

\begin{corollary}\label{coriterated}
Let $R_1,...,R_n$ be algebras graded by abelian groups $A_1,...,A_n$ respectively, each satisfying the finiteness conditions from Section \ref{sec:main}, and let $t=\{t_{ij}\}$ be a collection of bicharacters as above. The Hochschild cohomology of $R=R_1\otimes^t\cdots \otimes^tR_n$ can be decomposed
\[
   \HH^*(R,R)\cong \bigoplus_{a_1,\ldots,a_n} \bigotimes_{i=1}^{n}\HH^*(R_i, {}_{\hat{a}_{1}\cdots\hat{a}_{i-1}}(R_i)_{\hat{a}_{i+1}\cdots\hat{a}_{n}})^{a_i}.
\]
The cup product and Gerstenhaber bracket on $\HH^*(R,R)$ can be computed from that of the factors in this decomposition, in a similar way to Theorems \ref{cup-prod-thm} and  \ref{thm:G-alg}.
\end{corollary}

We leave the last statement to be interpreted properly by the interested reader. We also skip the proof, since it is a simple induction applying Theorem  \ref{vect-space-theorem} repeatedly (using the fact that $R_1\otimes^t\cdots \otimes^tR_n$ can be viewed as an iterated twisted tensor product with two factors at a time, in a similar way to \cite[Lemma 5.1]{BO}).

As an application of Corollary \ref{coriterated} we compute the Hochschild cohomology of quantum complete intersections with more than two indeterminates.
This has yet wider applications; for example, many Nichols algebras arising in the theory of
pointed Hopf algebras have associated graded algebras that are quantum complete
intersections, and this structure can have important homological implications. 

Let $q$ be a collection of elements $q_{ij}$ in $k^\times$ for $i<j$, and set $t_{ij}:\Z\times \Z\to k^\times$ to be the bicharacter $t(a,b)=q_{ij}^{ab}$. Extending our earlier notation from
the case of two indeterminates to many, we set
\[
\Lambda_q(m_1,\ldots,m_n)=\Lambda(m_1)\otimes^t\cdots \otimes^t \Lambda(m_n).
\]
\begin{theorem}\label{iteratedqcithm}
Assume that the scalars $q_{ij}$ freely generate a free abelian subgroup of $k^\times$. Then as an algebra
\[
\HH^*(\Lambda_q(m_1,\ldots,m_n))\cong  k[U]/(U^2) \times_k {\textstyle \bigwedge^*_k}(V_1,\ldots,V_n) ,
\]
i.e.~it is the fiber product of $k[U]/(U^2)$, $U$ in degree $0$, with an exterior algebra $\bigwedge^*_k(V_1,\ldots,V_n)$, $V_i$ in degree $1$. The bracket is given by
\[
[V_i,U] = (m_i-1) U, \quad \text{and }\ [V_i,V_j]=0.
\]
\end{theorem}

The proof is very similar to that of Theorem \ref{qcibracketthm}, and so we will omit it. In the notation there, $U$ corresponds to $x^{m_1-1}\otimes\cdots \otimes x^{m_n-1}$ and each $V_i$ corresponds to $1\otimes \cdots \otimes x[1]\otimes\cdots \otimes 1$
(the $x[1]$ in the $i$th position).

In general, the description of Hochschild cohomology of $\Lambda_q(m_1,\ldots,m_n)$
depends on what kind of subgroup of $k^\times$ the scalars $q_{ij}$ generate. Various other cases are easy to compute as well, for example when all of the $q_{ij}$ are equal.

Bergh and Oppermann computed a \emph{part} of Hochschild cohomology in the case of many indeterminates, with its algebra structure. The authors of \cite{BE,BGMS,EF} do not treat the many indeterminate case; it likely would not have been feasible with the methods available. 
Thus Theorem \ref{iteratedqcithm} illustrates the usefulness of our main theorem.

\section{Skew group algebras}\label{sec:examples2}

In this section we give another large class of examples to which
our main theorem applies, namely skew group algebras for which
the group is abelian.
We treat first the case of a symmetric algebra with group action.

Assume that $k$ is algebraically closed of characteristic $0$, and let $G$ be a finite abelian group acting on a finite dimensional vector space $V$. This action extends naturally to an action on the symmetric algebra  $S(V)$. We can form the twisted group algebra $S(V)\rtimes G$. As a graded vector space this is the tensor product $S(V)\otimes kG$, and the multiplication is given by
\[
(f\otimes g)\cdot (f'\otimes g')= f g(f') \otimes gg' .
\]
Let us explain how to see $S(V)\rtimes G$ as a bicharacter twisted tensor product. Recall that $\widehat{G}=\Hom(G,k^\times)$ denotes the group of characters of $G$. There is a natural bicharacter $t:\widehat{G}\times G\to k^\times$, $t(\phi, g)=\phi(g)$. Given $\phi\in \widehat{G}$ we consider the eigenspace
 \[
 S(V)^{\phi}=\big\{ f : g(f)=\phi(g) f \text{ for all } g\in G\big\} .
 \]
Since $k$ is algebraically closed of characteristic $0$ we have a decomposition
\[
S(V)=\bigoplus_{\phi \in \widehat{G}}S(V)^\phi.
\]
In fact, this makes $S(V)$ into a $\widehat{G}$-graded algebra. The group algebra $kG$ is trivially $G$-graded. Putting all this structure together, we find that
\[
S(V)\rtimes G=S(V)\otimes^t kG.
\]
Using this observation we can recover---in the abelian case---a result Buchweitz, proved independently in work of Farinati~\cite{F} and 
Ginzburg and Kaledin~\cite{GK}. This can also be seen as resulting from a special case of a spectral sequence of Negron \cite{CN} (which degenerates for us, by the assumption on $k$).

\begin{corollary}\label{skewcor}
There is an isomorphism of Gerstenhaber algebras
\begin{align*}
    \HH^*(S(V)\rtimes G) & \cong\bigoplus_{g} \HH^*(S(V),S(V)_{\hat{g}})^G\\
    & \cong \HH^*(S(V),S(V)\rtimes G)^G . 
\end{align*}
\end{corollary}

\begin{proof}
Using our main theorem, 
\[
\HH^*(S(V)\rtimes G)\cong \bigoplus_{\phi,g} \HH^*(S(V),S(V)_{\hat{g}})^{\phi}\otimes \HH^*(kG,{}_{\phi}kG)^{g}.
\]
Since $kG$ is semisimple, $\HH^*(kG,kG) \cong kG$, and $\HH^*(kG,{}_{\phi}kG)=0$ if $\phi\neq 1$. From here the statement follows.
\end{proof}

 Shepler and the second author investigated the cup product structure of 
the Hochschild cohomology $\HH^*(S(V)\rtimes G)$ in \cite{SW}. Their description of this structure can be recovered by inspecting our proof above (but note that they work more generally with any finite group).
Similarly, Negron and the second author described Gerstenhaber brackets on
$\HH^*(S(V)\rtimes G)$ for any finite group $G$, and~\cite[Theorem~5.2.3]{NW2} can
be recovered, in case $G$ is abelian, from our Corollary~\ref{skewcor} above. 

On a related note, we remark that the algebra $\HH^*(S(V),S(V)\rtimes G)$ appearing above is the orbit Hochschild cohomology which we considered in Section \ref{sec:cup}.

\begin{remark}
More generally, one can consider an action of $G$ on any $k$-algebra $R$. 
If $G$ is abelian and $k$ has ``enough'' roots of unity, then the action can be diagonalized, and the skew group algebra $R\rtimes G$ can be realized as a bicharacter twisted tensor product, as it was for the case $R=S(V)$ above.  Therefore, we can use our decomposition theorem to compute the Hochschild cohomology of $R\rtimes G$.

More generally still, the same remarks apply when $G$ acts on an algebra $R$ and grades an algebra $S$. Then one can form a twisted tensor product algebra $R\rtimes S$
with twist determined by both the $G$-grading and the $G$-action; 
the product is given by $r \otimes s\cdot r'\otimes s' = r g(r')\otimes ss' $ when $s'\in S^g$. Under the same hypotheses as above one can replace this with a bicharacter twisted tensor product, and thereby compute its Hochschild cohomology.
\end{remark}

\end{document}